\documentclass[final,3p,times]{elsarticle}
\usepackage{amsmath,amssymb,amsfonts,epsf,amsthm}
\usepackage{hyperref}
\hypersetup{colorlinks=true, urlcolor=blue, citecolor=cyan, pdfborder={0 0 0},}
\usepackage{soul}
\usepackage{natbib}
\usepackage{multirow}

\usepackage[ruled, vlined, Algorithm]{algorithm}
\usepackage{algorithmic}

\newtheorem{theorem}{Theorem}[section]
\newtheorem{lemma}{Lemma}[section]
\newtheorem{proposition}{Proposition}[section]

\newtheorem{conclusion}{Conclusion}[section]

\newtheorem{remark}{Remark}[section]

\newtheorem{definition}{Definition}[section]

\numberwithin{equation}{section}

\usepackage{amssymb}
\journal{--}

\allowdisplaybreaks[1]

\begin{document}

\begin{frontmatter}

\title{  Metric dimension, minimal doubly resolving sets and the strong metric dimension for jellyfish graph  and cocktail party graph }

%\tnotetext[label1]{}
\author[label1]{Jia-Bao Liu}
\ead{liujiabaoad@163.com;liujiabao@ahjzu.edu.cn}
\author[label2]{Ali Zafari \corref{1}}
\ead{zafari.math.pu@gmail.com; zafari.math@pnu.ac.ir}
\author[label2]{Hassan Zarei}
\ead{hasan.zarei@pnu.ac.ir}
\address[label1]{School of Mathematics and Physics,
Anhui Jianzhu University, Hefei 230601, P.R. China}
\address[label2]{Department of Mathematics, Faculty of Science,
Payame Noor University, P.O. Box 19395-4697, Tehran, Iran}
\cortext[1]{Corresponding author}
%%%%%%%%%%%%%%%%%%%%%%%%%%%%%%
%%%%%%%%%%%%%%%%%%%%%%%%%%%%%%
\begin{abstract}
Let $\Gamma$ be a simple connected undirected graph with vertex set $V(\Gamma)$ and edge set $E(\Gamma)$.
The metric dimension of a graph $\Gamma$ is the least number of vertices in a set with the property that the
list of distances from any vertex to those in the set uniquely identifies that vertex. For an ordered subset
$W = \{w_1, w_2, ..., w_k\}$ of vertices in a  graph $\Gamma$ and a vertex $v$ of $\Gamma$, the metric representation
of $v$ with respect to $W$ is the $k$-vector $r(v | W) = (d(v, w_1), d(v, w_2), ..., d(v, w_k ))$. If every pair of distinct
vertices of $\Gamma$ have different metric representations then the ordered set $W$ is called a resolving set of $\Gamma$.
It is known that the problem of computing this invariant is NP-hard. In this paper, we consider the problem of determining
the cardinality $\psi(\Gamma)$ of minimal doubly resolving sets of $\Gamma$, and the strong metric dimension  for jellyfish
graph $JFG(n, m)$ and cocktail party graph $CP(k+1)$.
\end{abstract}
%%%%%%%%%%%%%%%%%%%%%%%%%%%%%%
%%%%%%%%%%%%%%%%%%%%%%%%%%%%%%
\end{frontmatter}
%%%%%%%%%%%%%%%%%%%%%%%%%%%%%%
%%%%%%%%%%%%%%%%%%%%%%%%%%%%%%
\section{Introduction}
\label{sec:introduction}
In this paper we consider finite, simple, and connected graphs. The vertex and edge sets of a graph $\Gamma$ are denoted by
$V(\Gamma)$ and $E(\Gamma)$, respectively. For $u, v \in V (\Gamma)$, the length of a shortest path from $u$ to $v$ is called
the distance between $u$ and $v$ and is denoted by $d_{\Gamma}(u, v)$, or simply  $d(u, v)$. The adjacency and non-adjacency
relations are denoted by $\sim$ and $\nsim$, respectively. The size of the largest clique in the graph $\Gamma$ is denoted by
$\omega(\Gamma)$ and the size of the largest independent sets of vertices by $\alpha(\Gamma)$. A vertex $x\in V(\Gamma)$ is said
to resolve a pair $u, v \in V(\Gamma)$ if $d_{\Gamma}(u, x)\neq d_{\Gamma}(v, x)$. For an ordered subset $W = \{w_1, w_2, ..., w_k\}$
of vertices in a connected graph $\Gamma$ and a vertex $v$ of $\Gamma$, the metric representation of $v$ with respect to $W$ is the
$k$-vector $r(v | W) = (d(v, w_1), d(v, w_2), ..., d(v, w_k ))$. If every pair of distinct vertices of $\Gamma$ have different
metric representations then the ordered set $W$ is called a resolving set of $\Gamma$. Indeed, the set $W$ is called a resolving set
for $\Gamma$ if $r(u | W) = r(v | W)$ implies that $u = v$ for all pairs $u, v$ of vertices of $\Gamma$. If the set $W$ is as small
as possible, then it is called a metric basis of the graph $\Gamma$. We recall that the  metric dimension of $\Gamma$, denoted by
$\beta(\Gamma)$ is defined as the minimum cardinality of a resolving set for $\Gamma$.If $\beta(\Gamma)=k$, then $\Gamma$ is said
to be $k$-dimensional. Chartrand et. al. ~\cite{e-1} determined the bounds of the metric dimensions for any connected graphs and
determined the metric dimensions of some well known families of graphs such as trees, paths, and complete graphs. Bounds on $\beta(\Gamma)$
are presented in terms of the order and the diameter of $\Gamma$. All connected graphs of order $n$ having metric dimension $1, n-1$, or $n-2$
are determined. The concept of resolving set and that of metric dimension date back to the 1950s. They were defined by Bluementhal ~\cite{b-1-a}
in the context of metric space. These notions were introduced to graph theory by Harary and Melter ~\cite{f-1} and Slater ~\cite{o-1} in the 1970s.
For more results related to these concepts see ~\cite{a-1-1,b-1-1,e-1-1,f-1-1,i-1-a,p-1}. These concepts  have different applications in the areas
of network discovery and verification ~\cite{b-1}, robot navigation ~\cite{g-1}, chemistry ~\cite{e-1}, and combinatorical optimization ~\cite{n-1}.
Notice, for each connected graph $\Gamma$ and each ordered set $W = \{w_1, w_2, ..., w_k\}$ of vertices of $\Gamma$, that the $i^{th}$ coordinate of
$r(w_i | W)$ is $0$ and that the $i^{th}$ coordinate of all other vertex representations is positive. Thus, certainly $r(u | W) = r(v | W)$ implies
that $u = v$ for $u\in W$. Therefore, when testing whether an ordered subset $W$ of $V(\Gamma)$ is a resolving set for $\Gamma$, we need only be
concerned with the vertices of $V(\Gamma)-W$.

C\'{a}ceres et al.~\cite{c-1} define the notion of a doubly resolving set as follows. Vertices $x, y$ of the graph $\Gamma$ of order at least 2, are
said to doubly resolve vertices $u, v$ of $\Gamma$ if $d(u, x) - d(u, y) \neq d(v, x) - d(v, y)$. A  set $Z = \{z_1, z_2, ..., z_l\}$ of vertices of $\Gamma$ is a doubly resolving set of $\Gamma$ if every two distinct vertices of $\Gamma$ are doubly resolved by some two vertices of $Z$. The minimal doubly resolving set is a doubly resolving set with minimum cardinality. The cardinality of minimum doubly resolving set is denoted by  $\psi(\Gamma)$.
The minimal doubly resolving sets for Hamming and Prism graphs has been obtained in ~\cite{i-1} and ~\cite{d-1}, respectively. Another researchers
in ~\cite{a-1} determined the minimal doubly resolving sets for necklace graph. Since if $x, y$ doubly resolve $u, v$, then$d(u, x) - d(v, x) \neq 0$
or $d(u, y) - d(v, y) \neq 0$, and hence $x$ or $y$ resolve $u, v$. Therefore, a doubly resolving set is also a resolving set and
$\beta(\Gamma) \leq\psi(\Gamma)$.

The strong metric dimension problem was introduced by A. Seb\"{o} and E. Tannier ~\cite{n-1} and further investigated by O. R. Oellermann and J. Peters-Fransen ~\cite{l-1}. Recently, the strong metric dimension of distance hereditary graphs has been studied by T. May and O. R. Oellermann
~\cite{j-1}. For more results related to this concept see ~\cite{q-1-a}. A vertex $w$ strongly resolves two vertices $u$ and $v$ if $u$ belongs
to a shortest $v - w$ path or $v$ belongs to a shortest $u - w$ path. A set $N= \{n_1, n_2, ..., n_m\}$ of vertices of $\Gamma$ is a strong resolving
set of $\Gamma$ if every two distinct vertices of $\Gamma$ are strongly resolved by some vertex of $N$. A strong metric basis of $\Gamma$ is a strong resolving set of the minimum cardinality. Now, the strong metric dimension of $\Gamma$, denoted by $sdim(\Gamma)$ is defined as the cardinality of its strong metric basis. It is easy to see that if a vertex $w$ strongly resolves vertices $u$ and $v$ then $w$ also resolves these vertices. Hence every
strong resolving set is a resolving set and $\beta(\Gamma) \leq sdim(\Gamma)$.

All three previously defined problems are NP-hard in general case. The proofs of NP-hardness are given for the metric dimension problem in ~\cite{g-1},
for the minimal doubly resolving set problem in ~\cite{h-1} and for the strong metric dimension problem in ~\cite{l-1}. Intrinsic metrics on a graph
have become of interest, as generally discussed in ~\cite{a-1-a,a-1-1,f-1-2,i-1,k-1-1,m-1}, for some classes of graphs. In this paper, we  consider the problem of determining the cardinality $\psi(\Gamma)$  of minimal doubly resolving sets of $\Gamma$, and the strong metric dimension  for jellyfish graph $JFG(n, m)$ and cocktail party graph $CP(k+1)$. In fact, we find the minimum cardinality of resolving set, doubly resolving set  and strong resolving
set of jellyfish graph $JFG(n, m)$. In particular, we show that  the minimum cardinality of resolving set, doubly resolving set  and strong resolving
set of cocktail party graph $CP(k+1)$ is $k+1$.
%%%%%%%%%%%%%%%%%%%%%%%%%%%%%%
%%%%%%%%%%%%%%%%%%%%%%%%%%%%%%
\section{Definitions And Preliminaries}
%%%%%%%%%%%%%%%%%%%%%%%%%%%%%%
\begin{definition} \label{b.1}~\cite{book1}
Let $\Gamma$ be a graph with automorphism group $Aut(\Gamma)$. We say that $\Gamma$ is  vertex transitive  graph if, for any vertices $x, y$ of $\Gamma$  there is some  $\varphi$ in $Aut (\Gamma)$,  such that $\varphi(x) = y$. Also,  we say that $\Gamma$ is symmetric if, for all vertices $u, v, x, y$ of $\Gamma$ such that $u$ and $v$ are adjacent, also, $x$ and $y$ are adjacent, there is an automorphism $\varphi$ such that $\varphi(u)=x$ and $\varphi(v)=y$. Finally, we say that $\Gamma$ is distance transitive if, for all vertices $u, v, x, y$ of $\Gamma$ such that $d(u, v)=d(x, y)$ there is an automorphism $\varphi$ such that $\varphi(u)=x$ and $\varphi(v)=y$.
\end{definition}
%%%%%%%%%%%%%%%%%%%%%%%%%%%%%%
\begin{definition}\label{b.2}~\cite{book1}
Let $G$ be a finite group and $\Omega$ a subset of $G$ that is closed under taking inverses and does not contain the identity. A Cayley graph
$\Gamma=Cay(G, \Omega)$ is a graph whose vertex set and edge set are defined as follows:
$$V (\Gamma) =G; \,\, \,\ E(\Gamma) = \{\{x, y\} \, | \,\, x^{-1}y \in \Omega\}.$$
\end{definition}
%%%%%%%%%%%%%%%%%%%%%%%%%%%%%%
\begin{remark} \label{b.3}~\cite{book1}
Let $\Gamma$ be a graph. It is clear that we have a hierarchy of  the conditions is

$\textbf{distance transitive}\Rightarrow \textbf {symmetric}\Rightarrow \textbf {vertex  transitive}$
\end{remark}
%%%%%%%%%%%%%%%%%%%%%%%%%%%%%%
\begin{proposition}~\cite{k-1}\label{b.4}
Let $\Gamma=Cay(\mathbb{Z}_n, S_k)$ be the Cayley graph on the  cyclic group $\mathbb{Z}_n$ $(n\geq 4),$ where
$S_1=\{1, n-1\}$, ..., $S_k=S_ {k-1}\cup\{k, n-k\}$ are the inverse closed subsets of $\mathbb{Z}_n-\{0\}$ for any $k\in \mathbb{N}$, $1\leq k\leq [\frac{n}{2}]-1$. Then $\chi(\Gamma) = \omega(\Gamma)=k+1$ if and only if $k+1|n,$ where the chromatic number $\chi(\Gamma)$  of $\Gamma$  is the minimum number $k$ such that $\Gamma$ is $k$ colorable.
\end{proposition}
%%%%%%%%%%%%%%%%%%%%%%%%%%%%%%
\begin{definition}~\cite{f-1-2} \label{b.7}
Let $\Gamma$ be a graph, and let $W = \{w_1, ... ,w_k\} \subseteq V(\Gamma)$. For each vertex $v \in V(\Gamma)$, the adjacency representation
of $v$ with respect to $W$ is the $k$-vector $$\hat{r}(v|W) = (a_\Gamma(v,w_1), ...,  a_\Gamma(v,w_k)),$$ where
\begin{equation*}
a_\Gamma(v, w_i)= \left\{
\begin{array}{lr}
0 \,\,\,\,\,\,\,\,\,\,\,\,\,\,\,\ if \,\ v=w_i, \\

1 \,\,\,\,\,\,\,\,\,\,\,\,\,\,\,\ if \,\ v\sim w_i, \\

2 \,\,\,\,\,\,\,\,\,\,\,\,\,\,\,\ if \,\ v\nsim w_i. \\
\end{array} \right.
\end{equation*}
The set $W$ is an adjacency resolving set for $\Gamma$ if the vectors $\hat{r}(v|W)$ for $v \in V (\Gamma)$ are distinct. The
minimum cardinality of an adjacency resolving set is the adjacency dimension of $\Gamma$, denoted by $\hat{\beta}(\Gamma)$.
An adjacency resolving set of cardinality $\hat{\beta}(\Gamma)$ is an adjacency basis of $\Gamma$.
\end{definition}
%%%%%%%%%%%%%%%%%%%%%%%%%%%%%%
\begin{lemma}~\cite{f-1-2}\label{b.8}
Let $\Gamma$ be a graph of order $n$.\newline

1) If $diam(\Gamma) = 2$, then $\hat{\beta}(\Gamma)=\beta(\Gamma)$.\newline

2) If $\Gamma$ is connected, then $\beta(\Gamma)\leq \hat{\beta}(\Gamma)$.\newline

3) $1\leq \hat{\beta}(\Gamma) \leq n-1$.\newline
\end{lemma}
%%%%%%%%%%%%%%%%%%%%%%%%%%%%%%
%%%%%%%%%%%%%%%%%%%%%%%%%%%%%%
\section{Main results}
\noindent \textbf{Metric dimension, minimal doubly resolving sets and the strong metric dimension  of  jellyfish graph $JFG(n, m)$ }\\

For an integer $n\geq 3$, the $n$-cycle is the graph $C_n$ with vertex set $V(C_n)=\{v_1, v_2,...,v_n\}$  or simply $\{1, 2, ..., n\}$ and edge set $E(C_n)=\{v_iv_{i+1}:1\leq i\leq n\}$, where $v_{n+1}=v_1$. An interesting family of  graphs  of order $nm+n$ is defined  as follows. Let $\Gamma$
be a graph with vertex set $V_1 \cup V_2$, where $V_1=V(C_n)$,   $V_2=\{A_{1j}, A_{2j}, ..., A_{nj}\}$, and let   $A_{ij}=\cup_{j=1} ^ m v_{ij}$,
$1\leq i \leq n$. Suppose that every vertex $i\in V_1$  is adjacent to vertices $v_{i1}, v_{i2}, ..., v_{im}\in A_{ij}$, and $deg(v_{ij})=1$ for
every  vertex $v_{ij}\in A_{ij}$, then the resulting graph is called the jellyfish graph $JFG(n, m)$ with parameters $m$ and $n$. In particular,
if $n$ is an even integer then  the  jellyfish graph $JFG(n, m)$  is a bipartite graph. In this paper, we consider the problem of determining the cardinality $\psi(JFG(n, m))$  of minimal doubly resolving sets of  the  jellyfish graph $JFG(n, m)$. First, we find the metric dimension of
jellyfish graph $JFG(n, m)$, in fact we prove that  if $n\geq3$  and $m\geq2$ then the metric dimension of  jellyfish graph $JFG(n, m)$ is  $nm-n$.
Also, we consider the problem of determining the cardinality $\psi(JFG(n, m))$  of minimal doubly resolving sets of $JFG(n, m)$, and the strong metric
dimension  of the jellyfish graph $JFG(n, m)$. Moreover, we find an  adjacency dimension  of the jellyfish graph $JFG(n, m)$.
%%%%%%%%%%%%%%%%%%%%%%%%%%%%%%
\begin{theorem}\label{f.1}
Let $n, m$ be integers such that $n\geq 3$, $m\geq 2$. Then the metric dimension of  jellyfish graph $JFG(n, m)$ is  $nm-n$.
\end{theorem}
\begin{proof}
Let $V(JFG(n, m))=V_1 \cup V_2$, where $V_1=V(C_n)=\{1, 2, ..., n\}$,   $V_2=\{A_{1j}, A_{2j}, ..., A_{nj}\}$, and let
$A_{ij}=\cup_{j=1} ^ m v_{ij}$, $1\leq i \leq n$. Suppose that every vertex $i\in V_1$  is adjacent to vertices
$v_{i1}, v_{i2}, ..., v_{im}\in A_{ij}$. We can show that the diameter of  jellyfish graph $JFG(n, m)$ is $[\frac{n}{2}]+2$.
 In the following cases, we show that the metric dimension of  jellyfish graph $JFG(n, m)$ is $nm-n$.\newline

Case 1. Let $W$  be  an ordered subset  of $V_1$ in the  jellyfish graph $JFG(n, m)$ such that $|W|\leq n$.
It is easy to prove that  if  $|W|< n$, then $W$ is  not a resolving set of jellyfish graph $JFG(n, m)$. In particular, if $|W|= n$ then
we show that $W$ is not a resolving set of   jellyfish graph $JFG(n, m)$. We may assume without loss of generality that an ordered subset
is $W=\{1,2, ..., n\}$. Hence, $V(JFG(n, m))- W=\{A_{1j}, A_{2j}, ..., A_{nj}\}$. Therefore, the metric representation of  the  vertices
$v_{11}, v_{12}, ..., v_{1m}\in A_{1j}$ with respect to $W$ is the same as $n$-vector. Thus, $W$ is not a resolving set of   jellyfish graph
$JFG(n, m)$. \newline

Case 2.  Let $W$  be an ordered subset of $V_2$ in the   jellyfish graph $JFG(n, m)$ such that $W=\{A_{2j}, A_{3j}, ..., A_{nj}\}$. Hence,
$V(JFG(n, m))- W=\{1, 2, ..., n, A_{1j}\}$. We know that $|W|= nm-m$. Therefore, the metric representation of  the  vertices
$v_{11}, v_{12}, ..., v_{1m}\in A_{1j}$ with respect to $W$ is the same as $nm-m$-vector. Thus, $W$ is not a resolving set of  jellyfish graph
$JFG(n, m)$. \newline

Case 3. Let $W$  be an ordered subset of $V_2$ in the  jellyfish graph $JFG(n, m)$ such that $W=\{A_{1j}, A_{2j}, A_{3j}, ..., A_{nj}-\{v_{n1} , v_{n2}\}\}$. Hence, $V(JFG(n, m))- W=\{1, 2, ..., n, v_{n1},  v_{n2} \}$. We know that $|W|= nm-2$. Therefore, the metric representation of  the
vertices  $v_{n1}, v_{n2}\in A_{nj}$ with respect to $W$ is the same as $nm-2$-vector. Thus, $W$ is not a resolving set of   jellyfish graph
$JFG(n, m)$. \newline

Case 4. Let $W$  be an ordered subset of $V_2$ in  the jellyfish graph $JFG(n, m)$ such that $|W|= nm-1$. We show that $W$ is  a resolving set of   jellyfish graph $JFG(n, m)$. We may assume without loss of generality that an ordered subset is $W=\{A_{1j}, A_{2j}, ..., A_{nj}-v_{nm}\}$. Hence,
$V(JFG(n, m))- W=\{1, 2, ..., n, v_{nm}\}$. We can show that all the vertices $1, 2, ..., n, v_{nm}\in V(JFG(n, m))-W$ have different
representations with respect to $W$. Because, for every $k\in V(JFG(n, m))- W$, $1\leq k\leq n$ and $v_{ij}\in A_{ij}$, $1\leq i \leq n$,
$1\leq j \leq m$, if $k=i$ then we have $d(k, v_{ij})=1$, otherwise $d(k, v_{ij})>1$. Also, for the vertex $v_{nm} \in V(JFG(n, m))- W$ with
$v_{nm}\neq v_{ij}\in A_{ij}$, $1\leq i \leq n$, $1\leq j \leq m$, if $i=n$ then we have $d(v_{nm}, v_{ij})=2$, otherwise
$d(v_{nm}, v_{ij})>2$. Therefore, all the vertices $1, 2, ..., n, v_{nm}\in V(JFG(n, m))-W$ have different representations with respect to $W$.
This implies that $W$ is a resolving set of  jellyfish graph $JFG(n, m)$.
\newline

Case 5. Let $W$  be an ordered subset of $V_2$ in  the  jellyfish graph $JFG(n, m)$ such that $W=\{A_{1j}-v_{1m}, A_{2j}-v_{2m}, ..., A_{nj}-v_{nm}\}$.
Hence, $V(JFG(n, m))- W=\{1, 2, ..., n, v_{1m}, v_{2m}, ..., v_{nm}\}$. We know that $|W|= nm-n$. In a similar fashion which is done in Case 4, we can
show that all the vertices $1, 2, ..., n, v_{1m}, v_{2m}, ..., v_{nm}\in V(JFG(n, m))-W$ have different representations with respect to $W$. This implies that $W$ is a resolving set of   jellyfish graph $JFG(n, m)$.
\newline

Case 6. In particular,  let $W$  be an ordered subset of $V_2$ in  the  jellyfish graph $JFG(n, m)$ such that $|W|= nm$. We show that $W$ is
a resolving set of   jellyfish graph $JFG(n, m)$. We may assume without loss of generality that an ordered subset is $W=\{A_{1j}, A_{2j}, ..., A_{nj}\}$.   Hence $V(JFG(n, m))- W=\{1, 2, ..., n\}$. We can show that all the vertices $1, 2, ..., n\in V(JFG(n, m))-W$ have different representations with
respect to $W$. This implies that $W$ is a resolving set of   jellyfish graph $JFG(n, m)$.
\newline

From  the above cases, we  conclude that the minimum cardinality of a resolving set of the  jellyfish graph $JFG(n, m)$ is $nm-n$.
\end{proof}
%%%%%%%%%%%%%%%%%%%%%%%%%%%%%%
\begin{lemma}\label{f.4-1}
Let $n, m$ be integers such that $n\geq 3$, $m\geq 2$. Then the subset $Z=\{A_{1j}-v_{1m}, A_{2j}-v_{2m}, ..., A_{nj}-v_{nm}\}$ of vertices in  the  jellyfish graph $JFG(n, m)$ is not a doubly resolving set of  jellyfish graph $JFG(n, m)$.
\end{lemma}
\begin{proof}
We know that the ordered subset $Z=\{A_{1j}-v_{1m}, A_{2j}-v_{2m}, ..., A_{nj}-v_{nm}\}$ of vertices in  the  jellyfish graph $JFG(n, m)$ is a resolving
set of jellyfish graph $JFG(n, m)$ of size $nm-n$. Also by Theorem \ref{f.1},  the metric dimension of jellyfish graph $JFG(n, m)$ is
$\beta(JFG(n, m))=nm-n$. Moreover,  $B(JFG(n, m))\leq \psi(JFG(n, m))$. We show  that the subset $Z=\{A_{1j}-v_{1m}, A_{2j}-v_{2m}, ..., A_{nj}-v_{nm}\}$
of vertices in jellyfish graph $JFG(n, m)$ is not a doubly resolving set of jellyfish graph $JFG(n, m)$. Because, if  $u=v_{im}$ and
$v=i$, $1\leq i\leq n$, then for every $x,y\in Z$, we have $d(u, x) - d(u, y) = d(v, x) - d(v, y)$.
\end{proof}
%%%%%%%%%%%%%%%%%%%%%%%%%%%%%%
\begin{lemma}\label{f.4-2}
Let $n, m$ be integers such that $n\geq 3$, $m\geq 2$. Then the subset $Z=\{A_{1j}, A_{2j}, ..., A_{nj}-v_{nm}\}$ of vertices in  the  jellyfish graph $JFG(n, m)$ is not a doubly resolving set of jellyfish graph $JFG(n, m)$.
\end{lemma}
\begin{proof}
We show  that  subset $Z=\{A_{1j}, A_{2j}, ..., A_{nj}-v_{nm}\}$ of vertices in  the  jellyfish graph $JFG(n, m)$ is not a doubly resolving set of    jellyfish graph $JFG(n, m)$. Because, if  $u=v_{nm}$ and $v=n$, then for every $x,y\in Z$, we have $d(u, x) - d(u, y) = d(v, x) - d(v, y)$.
\end{proof}
%%%%%%%%%%%%%%%%%%%%%%%%%%%%%%
\begin{theorem}\label{f.4}
 Let $n, m$ be integers such that $n\geq 3$, $m\geq 2$. Then the cardinality of minimum doubly resolving set of jellyfish graph $JFG(n, m)$ is  $nm$.
\end{theorem}
\begin{proof}
Let $V(JFG(n, m))=V_1 \cup V_2$, where $V_1=V(C_n)=\{1, 2, ..., n\}$,   $V_2=\{A_{1j}, A_{2j}, ..., A_{nj}\}$, and let  $A_{ij}=\cup_{j=1} ^ m v_{ij}$,
$1\leq i \leq n$. Suppose that every vertex $i\in V_1$  is adjacent to vertices $v_{i1}, v_{i2}, ..., v_{im}\in A_{ij}$.
We know that the ordered subset $Z=\{A_{1j}, A_{2j}, ..., A_{nj}\}$ of vertices in the  jellyfish graph $JFG(n, m)$ is a resolving set of jellyfish
graph $JFG(n, m)$ of size $nm$. Also by Theorem \ref{f.1},  the metric dimension of   jellyfish graph $JFG(n, m)$ is $\beta(JFG(n, m))=nm-n$. Moreover,
$B(JFG(n, m))\leq \psi(JFG(n, m))$. We show  that the  subset $Z=\{A_{1j}, A_{2j}, ..., A_{nj}\}$ of vertices in jellyfish graph $JFG(n, m)$ is  a
doubly resolving set of  jellyfish graph $JFG(n, m)$. It is sufficient to show that for two vertices $u$ and $v$ of   jellyfish graph $JFG(n, m)$
there are vertices $x, y \in Z$ such that $d(u, x) - d(u, y) \neq d(v, x) - d(v, y)$.
Consider two vertices $u$ and $v$ of jellyfish graph $JFG(n, m)$. Then we have the following: \newline

Case 1. Let  $u\notin Z$ and $v\notin Z$. Hence, $u,v\in V_1=V(C_n)=\{1, 2, ..., n\}$. We can assume without loss of generality that
$u=i$ and $v=j$, $1\leq i, j\leq n$ and $i\neq j$. Therefore, if  $x=v_{i1}$ and $y=v_{j1}$, then we have $d(u, x) - d(u, y) \neq d(v, x) - d(v, y)$, because  $d(u, x) - d(u, y)<0$ and $d(v, x) - d(v, y)>0$. \newline

Case 2. Let  $u\in Z$ and $v\in Z$. Hence, $u,v\in V_2=\{A_{1j}, A_{2j}, ..., A_{nj}\}$. Therefore, if  $x=u$ and $y=v$, then we have
$d(u, x) - d(u, y) \neq d(v, x) - d(v, y)$, because  $d(u, x) - d(u, y)<0$ and $d(v, x) - d(v, y)>0$. \newline

Case 3. Finally, let  $u\notin Z$ and $v\in Z$. Hence, $u\in V_1=V(C_n)=\{1, 2, ..., n\} $ and $v\in V_2=\{A_{1j}, A_{2j}, ..., A_{nj}\}$.
We can assume without loss of generality that $u=k$, $1\leq k\leq n$  and $v=v_{11}\in A_{11}$. Therefore, if  $x=v_{k2}$ and $y=v_{11}$,
then we have $d(u, x) - d(u, y) \neq d(v, x) - d(v, y)$, because  $d(u, x) - d(u, y)\leq0$ and $d(v, x) - d(v, y)>0$. \newline

Thus, by Lemma \ref{f.4-1}, Lemma \ref{f.4-2}, and  the above cases we  conclude that the  cardinality of minimum doubly resolving set of
jellyfish graph $JFG(n, m)$ is $nm$.
\end{proof}
%%%%%%%%%%%%%%%%%%%%%%%%%%%%%%
\begin{lemma}\label{f.5}
Let $n, m$ be integers such that $n\geq 3$, $m\geq 2$. Then the subset $N=\{A_{1j}-v_{1m}, A_{2j}-v_{2m}, ..., A_{nj}-v_{nm}\}$ of vertices in the
jellyfish graph $JFG(n, m)$ is  not a strong resolving set of  jellyfish graph $JFG(n, m)$.
\end{lemma}
\begin{proof}
Let $M=V_2-N=\{v_{1m}, v_{2m}, ..., v_{nm}\}$, where $V_2$ is the set which is defined already. It is not hard to see that for every two
distinct vertices $u, v\in M$ there is not a vertex $w\in N$ such that $u$ belongs to a shortest $v - w$ path or $v$ belongs to a shortest
$u - w$ path. So, the subset $N=\{A_{1j}-v_{1m}, A_{2j}-v_{2m}, ..., A_{nj}-v_{nm}\}$ of vertices in jellyfish graph $JFG(n, m)$ is not
a strong resolving set of jellyfish graph $JFG(n, m)$. We conclude that if   $N$ is a strong resolving set of  jellyfish graph $JFG(n, m)$
then $|N|\geq nm-1$, because $|M|$ must be less than $2$.
\end{proof}
%%%%%%%%%%%%%%%%%%%%%%%%%%%%%%
\begin{theorem}\label{f.6}
Let $n, m$ be integers such that $n\geq 3$, $m\geq 2$. Then the strong metric dimension of  jellyfish graph $JFG(n, m)$ is $nm-1$.
\begin{proof}
By Lemma \ref{f.5}, we know that if $N$ is a strong resolving set of the jellyfish graph $JFG(n, m)$ then $|N|\geq nm-1$.
We show that the subset $N=\{A_{1j}, A_{2j}, ..., A_{nj}-v_{nm}\}$ of vertices in   jellyfish graph $JFG(n, m)$
is a strong resolving set of  jellyfish graph $JFG(n, m)$. It is sufficient to prove that  every two distinct vertices
$u,v \in V(JFG(n, m))-N=\{1,2, ...,n, v_{nm}\}$ is strongly resolved by a vertex $w\in N$. In the following cases we
show that the  strong metric dimension of  jellyfish graph $JFG(n, m)$ is $nm-1$.\newline

Case 1. Let $u$ and $v$ be two distinct vertices in $V(JFG(n, m))-N$ such that  $u, v\in V_1=V(C_n)=\{1, 2, ..., n\}$. So, there is  $i, j\in V_1$
such that $u=i$ and $v=j$. Therefore $i$ and $j$ will be strongly resolved by some $v_{ij}\in A_{ij}$, because $i$ and $v_{ij}$ are adjacent, and
hence $i$ belongs to a shortest $v_{ij} - j$ path. \newline

Case 2. Now, let $u$ and $v$ be two distinct vertices in $V(JFG(n, m))-N$ such that $u\in V_1=\{1, 2, ..., n\}$ and $v=v_{nm}$.
Without loss of generality we may assume $u=i$, where $i\in V_1$. Therefore $i$ and $v_{nm}$ will be strongly resolved by some $v_{ij}\in A_{ij}$,
because $i$ and $v_{ij}$ are adjacent, and hence $i$ belongs to a shortest $v_{ij} -v_{nm}$ path. \newline

From the above cases, we  conclude that the minimum cardinality of a strong metric dimension of the jellyfish graph $JFG(n, m)$ is $nm-1$.
\end{proof}
\end{theorem}
%%%%%%%%%%%%%%%%%%%%%%%%%%%%%%
\begin{lemma}\label{f.7}
Let $n, m$ be integers such that $n\geq 3$, $m\geq 2$. Then an ordered subset $W=\{A_{1j}-v_{1m}, A_{2j}-v_{2m}, ..., A_{nj}-v_{nm}\}$ of vertices in the
jellyfish graph $JFG(n, m)$ is  not the adjacency resolving set of  jellyfish graph $JFG(n, m)$.
\end{lemma}
\begin{proof}
Let $\Gamma= JFG(n,m)$ and  $M=V_2-W=\{v_{1m}, v_{2m}, ..., v_{nm}\}$, where $V_2$ is the set which is defined already. Thus, the adjacency  representation of  the  vertices  $v_{1m}, v_{2m}, ..., v_{nm}\in V(JFG(n, m))-W$ with respect to $W$ is the $nm-n$-vector
$\hat{r}(v_{1m}|W) =\hat{r}(v_{2m}|W)=...=\hat{r}(v_{nm}|W)=  (2, 2, ..., 2)$. Because, for every vertex
$w \in W$ we have $a_\Gamma(w, v_{1m}) = a_\Gamma(w, v_{2m}) = ...=a_\Gamma(w, v_{nm})=2$. We conclude that if $W$ is an  adjacency resolving set of  jellyfish graph $JFG(n, m)$ then $|W|\geq nm-1$, because $|M|$ must be less than $2$.
\end{proof}
%%%%%%%%%%%%%%%%%%%%%%%%%%%%%%
\begin{theorem}\label{f.8}
Let $n, m$ be integers such that $n\geq 3$, $m\geq 2$. Then the adjacency dimension of jellyfish graph $JFG(n, m)$ is $nm-1$.
\end{theorem}
\begin{proof}
By Lemma \ref{f.7}, we know that if $W$ is an adjacency resolving set of the jellyfish graph $JFG(n, m)$
then $|W|\geq nm-1$. Now, let $\Gamma= JFG(n,m)$ and $W$  be an ordered subset of $V_2$ in jellyfish graph $JFG(n, m)$ such that $|W|= nm-1$.
We show that $W$ is  an adjacency resolving set of jellyfish graph $JFG(n, m)$.
We may assume without loss of generality that  an ordered subset is $W=\{A_{1j}, A_{2j}, ..., A_{nj}-v_{nm}\}$. Hence
$V(JFG(n, m))- W=\{1, 2, ..., n, v_{nm}\}$. We can show that all the vertices $1, 2, ..., n, v_{nm}\in V(JFG(n, m))-W$ have different
adjacency representations with respect to $W$. Because, for every $k\in V(JFG(n, m))- W$, $1\leq k\leq n$ and
$ v_{ij}\in A_{ij}$, $1\leq i \leq n$, $1\leq j \leq m$, if $k=i$ then we have $a_\Gamma(k, v_{ij})=1$, otherwise $a_\Gamma(k, v_{ij})=2$.
Also, for the vertex $v_{nm} \in V(JFG(n, m))- W$ with $v_{nm}\neq v_{ij}\in A_{ij}$, $1\leq i \leq n$, $1\leq j \leq m$,  we have
$a_\Gamma(v_{nm}, v_{ij})=2$. Therefore, all the vertices $1, 2, ..., n, v_{nm}\in V(JFG(n, m))-W$ have different  adjacency representations with
respect to $W$. This implies that $W$ is an adjacency resolving set of the  jellyfish graph $JFG(n, m)$. We know conclude that the minimum cardinality of the adjacency resolving set  of jellyfish graph $JFG(n, m)$ is $nm-1$.
\end{proof}
%%%%%%%%%%%%%%%%%%%%%%%%%%%%%%
%%%%%%%%%%%%%%%%%%%%%%%%%%%%%%
\noindent \textbf{Metric dimension, minimal doubly resolving sets and the strong metric dimension  of cocktail party graph $CP(k+1)$}\\

Let $\Gamma=Cay(\mathbb{Z}_n, S_k)$ be the Cayley graph on the cyclic additive group $\mathbb{Z}_n$, where
$S_1=\{1, n-1\}$, ..., $S_k=S_ {k-1}\cup\{k, n-k\}$ are the inverse closed subsets of $\mathbb{Z}_n-\{0\}$ for any $k\in \mathbb{N}$,
$1\leq k\leq [\frac{n}{2}]-1$. If $n$ is an even integer and $k=\frac{n}{2}-1$, then we can show that $Cay(\mathbb{Z}_n, S_k)$ is obtained from
the complete graph $K_{2(k+1)}$ by deleting a perfect matching, and hence $Cay(\mathbb{Z}_n, S_k)$ is isomorphic to the
cocktail party graph $CP(k+1)$. Moreover, we can show that $Cay(\mathbb{D}_{2n}, \Omega)$, where
$$\mathbb{D}_{2n}=<a,b \,\  | \,\, a^n=b^2=1  , ba=a^{n-1}b>,$$ is the dihedral group of order $2n$ and
$\Omega=\{a, a^2, ..., a^{n-1}, ab, a^{2}b, ... , a^{n-1}b\}$ is the  inverse closed subset of $\mathbb{D}_{2n}-\{1\}$, is isomorphic to the
cocktail party graph $CP(n)$. In this section, we consider the problem of determining the cardinality $\psi(\Gamma)$  of minimal
doubly resolving sets of $\Gamma$. First, we show that if  $n$ is an even integer and $k=\frac{n}{2}-1$,  then the metric dimension of
$\Gamma$ is $k+1$. Also, we prove that if  $n$ is an even integer and $k=\frac{n}{2}-1$, then every minimal resolving set of $\Gamma$ is also
a doubly resolving set, and, consequently, $\psi(\Gamma)$ is equal to the metric dimension of $\beta(\Gamma)$, which is known from the
literature. Moreover, we find an explicit expression for the strong metric dimension of $\Gamma$.
%%%%%%%%%%%%%%%%%%%%%%%%%%%%%%
\begin{theorem}\label{c.1}
Let $\Gamma=Cay(\mathbb{Z}_n, S_k)$ be the Cayley graph on the cyclic group $\mathbb{Z}_n$ $(n\geq 8),$ where
$S_1=\{1, n-1\}$, ..., $S_k=S_ {k-1}\cup\{k, n-k\}$ are the inverse closed subsets of $\mathbb{Z}_n-\{0\}$ for any
$k\in \mathbb{N}$, $1\leq k\leq [\frac{n}{2}]-1$. If $n$ is an even integer and $k=\frac{n}{2}-1$,  then the metric dimension of $\Gamma$ is $k+1$.
\end{theorem}
\begin{proof}
Let $V(\Gamma)=\{1, ..., n\}$ be the vertex set of $\Gamma$.
By  proof of Proposition  3.4 in ~\cite{k-1}, we know that the diameter of $\Gamma$ is 2. Hence,
for all the vertices $x, y \in V (\Gamma)$, the length of a shortest path from $x$ to $y$ is   $d(x, y)=1$ or $2$. Also, we know that
the size of  largest clique in the graph $\Gamma$ is  $k+1$. Now, let $W$  be an ordered subset of vertices  in the graph $\Gamma$ such that
$W$  is a clique in $\Gamma$. Indeed, for every $x, y \in W$, we have $d_\Gamma(x, y)=1$. In the following cases, we show that the metric
dimension of $\Gamma$ is $k+1$.\newline

Case 1. It is easy to see that if  $|W|\leq k$, then $W$ is  not a resolving set of $\Gamma$. In particular, if $|W|= k$, then we show that $W$ is not
a resolving set of $\Gamma$. We may assume without loss of generality that an ordered subset of vertices in the graph $\Gamma$ is
$W=\{1, 2, 3, ..., k\}$. Hence, $V(\Gamma)- W=\{k+1, k+2, ..., n\}$. On the other hand, by  proof of Proposition  3.2 in ~\cite{k-1}, we know
that  for every vertex  $x$ in $\Gamma$, there is  exactly one $y$ in $\Gamma$ such that $x^{-1}y=k+1$, that is $d(x, y)=2$. Hence, there are
vertices $k+2, k+3, ..., n-1$ in $V(\Gamma)- W$ such that $d(1, k+2)=2$, $d(2, k+3)=2$, ..., $d(k, n-1)=2$. Therefore, the metric representations
of  the vertices $k+2, k+3, ..., n-1\in V(\Gamma)-W$ with respect to $W$ are the $k$-vectors
$r(k+2  |  W)=(2, 1, 1, ..., 1)$, $r(k+3  |  W)=(1, 2, 1, ..., 1)$, ..., $r(n-1  |  W)=(1, 1, 1, ..., 2)$.
Moreover,  the metric representations of the vertices $k+1, n\in V(\Gamma)-W$ with respect to $W$ is the $k$-vector
$r(k+1  |  W)=r(n  |  W)=(1, 1, 1, ..., 1)$. Thus, $W$ is not a resolving set of $\Gamma$. Because, the metric representation of the vertices
 $k+1, n$ is the same as $k$-vector.\newline

Case 2. Let $W$  be a clique in the graph $\Gamma$ such that $x\in W$, and $|W|= k$. We know that there is exactly one $y$ in $V(\Gamma)- W$ such that $x^{-1}y=k+1$, that is $d(x, y)=2$.
In the following, we show that  an ordered subset $(W\cup y)$ of vertices in the graph $\Gamma$ is not a resolving set of $\Gamma$. In this case,  we may assume without loss of generality that an ordered subset of vertices in the graph $\Gamma$ is  $W=\{1, 2, ..., k\}$,  and let $x=1$, $y=k+2$. Therefore,
the metric representations of  the vertices $k+1, n\in V(\Gamma)-(W\cup y)$ with respect to $(W\cup y)$ is the $k+1$-vector
$r(k+1  |  W)=r(n  |  W)=(1, 1, 1, ...,1)$. Thus,
$(W\cup y)$ is not a resolving set of $\Gamma$.\newline

Case 3. Now, let $W$  be a clique in the graph $\Gamma$ such that $|W|= k+1$. We show that $W$ is
a resolving set of $\Gamma$.  We may assume without loss of generality that an ordered subset of vertices in the graph $\Gamma$ is
$W=\{1, 2, 3, ..., k, k+1\}$.  Hence $V(\Gamma)- W=\{k+2, k+3, ..., n\}$. Therefore, the metric representations of the vertices
$k+2, k+3, ..., n\in V(\Gamma)-W$ with respect to $W$ are the $(k+1)$-vectors
$r(k+2  |  W)=(2, 1, 1, ..., 1)$, $r(k+3  |  W)=(1, 2, 1, ..., 1)$, ..., $r(n  |  W)=(1, 1, 1, ..., 2)$. Thus,
all the vertices of $V(\Gamma)-W$ have different representations with respect to $W$. This implies that $W$ is
a resolving set of $\Gamma$. Moreover, the metric dimension of $\beta(\Gamma)\leq k+1$, because $\Gamma$  is a vertex transitive graph.\newline

Case 4. In particular, let $W$  be a clique in the graph $\Gamma$ such that $|W|= k+1$. We show that  for each  $x\in V(\Gamma)- W$, an ordered
subset $(W\cup x)$ of vertices in the graph $\Gamma$   is  also a resolving set of $\Gamma$. We may assume without loss of generality that an
ordered subset of vertices in the graph $\Gamma$ is $W=\{1, 2, 3, ..., k, k+1\}$, and let $x=k+2$. So,  $(W\cup x)=\{1, 2, 3, ..., k+1, k+2\}$.
Hence $V(\Gamma)- (W\cup x)=\{k+3, k+4, ..., n\}$. Therefore, the metric representations of the vertices
$k+3, k+4, ..., n\in V(\Gamma)- (W\cup x)$, with respect to $(W\cup x)$ are the $k+2$-vectors
$r(k+3  |  W)=(1, 2, 1, ..., 1, 1)$, $r(k+4  |  W)=(1, 1, 2, ..., 1, 1)$, ..., $r(n  |  W)=(1, 1, 1, ..., 2, 1)$. This implies that  $(W\cup x)$
is also a resolving set of $\Gamma$.\newline

From the above cases, we  conclude that the minimum cardinality of a resolving set of $\Gamma$ is $k+1$.
Moreover, it is well known that every Cayley graph is vertex transitive. Hence, the cardinality of every minimal resolving
set in  $\Gamma$ is $k+1$.
\end{proof}
%%%%%%%%%%%%%%%%%%%%%%%%%%%%%%
\begin{theorem}\label{c.2}
Let $\Gamma=Cay(\mathbb{Z}_n, S_k)$ be the Cayley graph on the cyclic group $\mathbb{Z}_n$ $(n\geq 8),$ where
$S_1=\{1, n-1\}$, ..., $S_k=S_ {k-1}\cup\{k, n-k\}$ are the inverse closed subsets of $\mathbb{Z}_n-\{0\}$ for any $k\in \mathbb{N}$,
$1\leq k\leq [\frac{n}{2}]-1$. If $n$ is an even integer and $k=\frac{n}{2}-1$,  then the  cardinality of minimum doubly
resolving set of $\Gamma$ is $k+1$.
\end{theorem}
\begin{proof}
In this Theorem, let $V(\Gamma)=\{v_1, v_2, ..., v_n\}$, where $v_i=i$ for $1\leq i \leq n$.
We know that the ordered subset $Z=\{1, 2, 3, ..., k, k+1\}$ of vertices in the graph $\Gamma$ is a resolving set for $\Gamma$ of size
$k+1$. Also, by Theorem \ref{c.1},  the metric dimension of $\Gamma$ is $\beta(\Gamma)=k+1$. Moreover,  $B(\Gamma)\leq \psi(\Gamma)$.
We show  that the subset $Z=\{1, 2, 3, ..., k, k+1\}$ of vertices in the graph $\Gamma$ is a doubly resolving set of $\Gamma$.
It is sufficient to show that for two vertices $v_i$ and $v_j$ of $\Gamma$ there are vertices $x, y \in Z$ such that
$d(v_i, x) - d(v_i, y) \neq d(v_j, x) - d(v_j, y) $. Consider two vertices $v_i$ and $v_j$ of $\Gamma$. We may assume that $i <j$.
In the following cases, we show that the  cardinality of minimum doubly
resolving set of $\Gamma$ is $k+1$.\newline

Case 1. If $1 \leq i < j \leq k+1$,  then $v_i, v_j\in Z$. So $d(v_i, v_j)=1$. We can assume that
$x=v_i\in Z$, and $y=v_j\in Z$. Hence, we have $d(v_i, x) - d(v_i, y) \neq d(v_j, x) - d(v_j, y)$, because $d(v_i, x) - d(v_i, y)<0$ and
$d(v_j, x) - d(v_j, y)>0$.\newline

Case 2. Let $1 \leq i \leq k+1 < j\leq n$.  Hence, $v_i\in Z$ and $v_j\notin Z$. Moreover, we know that $d(v_i, v_j)=1$ or $2$. In the following,
let $d(v_i, v_j)=1$. We may assume  that $v_i=1$ and $v_j=n$. Hence, by taking  $x=v_i\in Z$ and $y=k+1\in Z$, we have
$0=0-0=d(v_i, x) - d(v_i, y) \neq d(v_j, x) - d(v_j, y)=1-2=-1$. Thus, the vertices $x$ and $y$ of $Z$ doubly resolve $v_i, v_j$.
Now, let  $d(v_i, v_j)=2$. We may assume  that  $v_i=1$ and $v_j=k+2$. Hence,
by taking $x=1\in Z$ and $y=2\in Z$, we have $-1=0-1=d(v_i, x) - d(v_i, y) \neq d(v_j, x) - d(v_j, y)=2-1=1$.
Thus, the vertices $x$ and $y$ of $Z$ doubly resolve $v_i, v_j$.\newline

Case 3. Finally, let $1 <k+1< i< j\leq n$.  Hence, $v_i\notin Z$ and $v_j\notin Z$. Also, we know that $d(v_i, v_j)=1$.
On the other hand, by  proof of Proposition  3.2 in ~\cite{k-1}, we know that  for every vertex  $v_i\in V(\Gamma)-Z$, there is  exactly
one $x\in Z$ such that $d(v_i, x)=2$. Also for every vertex $v_j\in V(\Gamma)-Z$, there is  exactly one $y\in Z$ such that  $d(v_j, y)=2$.
Hence, $1=2-1=d(v_i, x) - d(v_i, y) \neq d(v_j, x) - d(v_j, y)=1-2=-1$.
Thus, the vertices $x$ and $y$ of $Z$ doubly resolve $v_i, v_j$.

From the above cases, we  conclude that the minimum cardinality of a doubly resolving set of $\Gamma$ is $k+1$.
\end{proof}
%%%%%%%%%%%%%%%%%%%%%%%%%%%%%%
\begin{theorem}\label{c.3}
Let $\Gamma=Cay(\mathbb{Z}_n, S_k)$ be the Cayley graph on the cyclic group $\mathbb{Z}_n$ $(n\geq 8),$ where
$S_1=\{1, n-1\}$, ..., $S_k=S_ {k-1}\cup\{k, n-k\}$ are the inverse closed subsets of $\mathbb{Z}_n-\{0\}$ for any $k\in \mathbb{N}$,
$1\leq k\leq [\frac{n}{2}]-1$. If  $n$ is an even integer and $k=\frac{n}{2}-1$,  then the strong metric dimension of $\Gamma$ is $k+1$.
\end{theorem}
\begin{proof}
In this Theorem, let $V(\Gamma)=\{v_1, v_2, ..., v_n\}$, where $v_i=i$ for $1\leq i \leq n$.
We know that the ordered subset $N=\{1, 2, 3, ..., k, k+1\}$ of vertices in the graph $\Gamma$ is a resolving set for $\Gamma$ of size $k+1$.
Also, by Theorem \ref{c.1}, the metric dimension of $\Gamma$ is $\beta(\Gamma)=k+1$. Moreover,  $B(\Gamma)\leq sdim(\Gamma)$. We show  that the
subset $N=\{1, 2, 3, ..., k, k+1\}$ of vertices in the graph $\Gamma$ is  a strong resolving set of $\Gamma$.
Consider two vertices $v_i$ and $v_j$ of $\Gamma$. Assume that $i <j$. It is sufficient to prove that  there exists a vertex $w\in N$  such that
$v_i$ belongs to a shortest $v_j - w$ path or $v_j$ belongs to a shortest $v_i - w$ path.
Let $1 <k+1< i< j\leq n$. Hence $v_i\notin N$ and $v_j\notin N$.  Moreover, we have $d(v_i, v_j)=1$.
On the other hand, by  proof of Proposition  3.2 in ~\cite{k-1}, we know that  for every vertex  $v_i\in V(\Gamma)-N$, there is  exactly one
$w\in N$ such that $v_i^{-1}w=k+1$, indeed $d(v_i, w)=2$. Hence, $d(v_j, w)=1$. So, $d(v_i, w)=d(v_i, v_j)+d(v_j, w)$, that is,  vertex $v_j$
belongs to a shortest $v_i - w$ path, and hence $w$ strongly resolves vertices $v_i$ and $v_j$.
Also for $v_i\in N $ or $v_j\in N $, vertex $v_i$ or vertex $v_j$ obviously strongly resolves pair $v_i, v_j$. Therefore, $N$ is a strong resolving
set. Thus, the minimum cardinality of a strong resolving set of $\Gamma$ is $k+1$.
\end{proof}
%%%%%%%%%%%%%%%%%%%%%%%%%%%%%%
\begin{conclusion}\label{c.4}
Let $CP(n)\cong Cay(\mathbb{D}_{2n}, \Omega)$ be the Cayley graph on  dihedral group $\mathbb{D}_{2n}$, where $\Omega$ is the inverse closed subset of $\mathbb{D}_{2n}-\{1\}$  which is defined already. Then  the minimum cardinality of resolving set,  doubly resolving set  and strong resolving
set of $CP(n)$ is $n$.
\end{conclusion}
%%%%%%%%%%%%%%%%%%%%%%%%%%%%%%
\section{Conclusion}
In this paper, we determined the minimum cardinality of resolving set, doubly resolving set  and strong resolving set  of jellyfish graph $JFG(n, m)$.
In particular, it has be shown that  the minimum cardinality of resolving set,  doubly resolving set  and strong resolving set of cocktail party graph $CP(k+1)$ is $k+1$.\newline
%%%%%%%%%%%%%%%%%%%%%%%%%%%%%%

\bigskip
{\footnotesize
%\noindent \textbf{Ethics approval and consent to participate}\\
%Not applicable.\\[2mm]
%\noindent \textbf{Consent for publication}\\
%Not applicable.\\[2mm]
\noindent \textbf{Data Availability}\\
No data were used to support this study.\\[2mm]
\noindent \textbf{Conflicts of Interest}\\
The authors declare that there are no conflicts of interest
regarding the publication of this paper.\\[2mm]
%\noindent \textbf{Funding}\\
%Not applicable.\\[2mm]
%\noindent \textbf{Authors' contributions}\\
%All authors contributed equally and significantly in this manuscript, and they read and approved the final manuscript.\\[2mm]
\noindent \textbf{Acknowledgements}\\
The work was partially supported by the National Natural Science Foundation of China under the grant No. 11601006, and China Postdoctoral
Science Foundation under grant no. 2017M621579, Postdoctoral Science Foundation of Jiangsu Province under grant no. 1701081B, and Project
of Anhui Jianzhu University under grant nos. 2016QD116 and 2017dc03. \\[2mm]
\noindent \textbf{Authors' informations}\\
\noindent Jia-Bao Liu${}^a$
(\url{liujiabaoad@163.com;liujiabao@ahjzu.edu.cn})\\
Ali Zafari${}^{b}$(\textsc{Corresponding Author})
(\url{zafari.math.pu@gmail.com}; \url{zafari.math@pnu.ac.ir})\\
\noindent Hassan Zarei${}^b$
(\url{hasan.zarei@pnu.ac.ir})
%Vahid Hedayati${}^{b,c}$ (\url{v.hedayati1367@gmail.com})\\
%\noindent Vahid Hedayati${}^{b,c}$ (\url{v.hedayati1367@gmail.com})
%Davoud Nazari Susahab${}^{b}$
%(\url{susahab@yahoo.com})
%Elham Khakbaz${}^{a}$
%(\url{ekhakbaz2@gmail.com})

\noindent ${}^{a}$School of Mathematics and Physics, Anhui Jianzhu University, Hefei 230601, P.R. China.\\
${}^{b}$Department of Mathematics, Faculty of Science,
Payame Noor University, P.O. Box 19395-4697, Tehran, Iran.
%${}^{c}$Department of Mathematics, Shahid Madani Educational Institution, Hamedan, Iran.
}\\

%\bigskip
{\footnotesize
%\noindent \textbf{References}\\

\bigskip
\end{document}